\documentclass[10pt,reqno]{amsart}
\usepackage[utf8]{inputenc}
\usepackage{amsopn,amssymb,mathrsfs,xcolor,bm,url,enumitem}
\usepackage[abbrev]{amsrefs} 
\usepackage[a4paper,lmargin=3cm,rmargin=3cm,tmargin=4cm,bmargin=4cm,marginparwidth=2.8cm,marginparsep=1mm]{geometry} 
\usepackage[bookmarks=true,hyperindex,pdftex,colorlinks,citecolor=blue]{hyperref} 
\usepackage{cleveref}

\newtheorem{theorem}{Theorem}[section]
\newtheorem{corollary}[theorem]{Corollary}
\newtheorem{lemma}[theorem]{Lemma}
\newtheorem{proposition}[theorem]{Proposition}

\theoremstyle{definition}

\newtheorem{example}[theorem]{Example}
\newtheorem{remark}[theorem]{Remark}

\newcommand{\ds}{\displaystyle}
\newcommand{\ep}{\varepsilon}
\newcommand{\ind}[1]{\ensuremath{\mbox{\boldmath{$1$}}_{#1}}}
\newcommand{\ip}[2]{\ensuremath{\left\langle{#1}\,|\,{#2}\right\rangle}}
\newcommand{\lint}[4]{\ensuremath{\int_{#1}^{#2}{#3}\:\mathrm{d}{#4}}}
\newcommand{\map}[3]{\ensuremath{{#1}:{#2}\to{#3}}}
\newcommand{\md}{\ensuremath{\mathrm{d}}}
\newcommand{\me}{\ensuremath{\mathrm{e}}}
\newcommand{\M}{\mathcal{M}}
\newcommand{\N}{\mathbb{N}}
\newcommand{\n}[1]{\ensuremath{\left\|{#1}\right\|}}
\newcommand{\ndot}{\ensuremath{\left\|\cdot\right\|}}
\newcommand{\pn}[2]{\ensuremath{\left\|{#1}\right\|_{#2}}}
\newcommand{\pndot}[1]{\ensuremath{\left\|\cdot\right\|_{#1}}}
\newcommand{\R}{\mathbb{R}}
\newcommand{\restrict}[1]{\ensuremath{\!\!\upharpoonright_{#1}}}
\newcommand{\set}[2]{\ensuremath{\left\{{#1}\;:\;\,{#2}\right\}}}
\newcommand{\tn}[1]{\ensuremath{\left|\kern-.9pt\left|\kern-.9pt\left|{#1}\right|\kern-.9pt\right|\kern-.9pt\right|}}
\newcommand{\tndot}{\ensuremath{\left|\kern-.9pt\left|\kern-.9pt\left|\cdot\right|\kern-.9pt\right|\kern-.9pt\right|}}
\newcommand{\ts}{\textstyle}

\DeclareMathOperator{\essvar}{ess\,var}
\DeclareMathOperator{\Lip}{Lip}
\DeclareMathOperator{\midd}{mid}
\DeclareMathOperator{\sgn}{sgn}
\DeclareMathOperator{\var}{var}

\numberwithin{equation}{section}

\allowdisplaybreaks

\title[A new convergence analysis of the Camassa-Holm equation]{A new convergence analysis of the particle method for the Camassa-Holm equation}

\author[L. \'O N\'araigh]{Lennon \'O N\'araigh}
\address[L. \'O N\'araigh]{School of Mathematics and Statistics, University College Dublin, Belfield, Dublin 4, Ireland}
\email{onaraigh@maths.ucd.ie}

\author[K. E. Pang]{Khang Ee Pang}
\address[K. E. Pang]{School of Mathematics and Statistics, University College Dublin, Belfield, Dublin 4, Ireland}
\email{khang-ee.pang@ucdconnect.ie}

\author[R. J. Smith]{Richard J. Smith}
\address[R. J. Smith]{School of Mathematics and Statistics, University College Dublin, Belfield, Dublin 4, Ireland}
\email{richard.smith@maths.ucd.ie}

\date{\today}

\begin{document}

\begin{abstract}
We present a new self-contained convergence analysis of the particle method that can be applied to a range of PDEs, including the Camassa-Holm equation. It is a development of the analysis of Chertock, Liu and Pendleton, which used compactness properties of spaces of functions having bounded variation. In our analysis we establish solutions by applying a metric Arzel\`a-Ascoli compactness result to a space of measure-valued functions equipped with the bounded Lipschitz metric. All the convergence and regularity results of the previous analysis follow as a consequence and are computationally easier to establish.
\end{abstract}

\maketitle

\section{Introduction and preliminaries}\label{sec:intro}

Fix a length scale $\alpha>0$. The \textit{Camassa-Holm Equation} is given by
\begin{equation}\label{eqn_CH}
 m_t + (um)_x + u_xm = 0 \qquad\text{where}\qquad m = (1-\alpha^2\partial_{xx})u \tag*{(CH)}
\end{equation}
in the region $\Omega := \R^+ \times \R$ and subject to the initial condition
\begin{equation}\label{eqn_init_condition}
u(0,\cdot) = \map{u_0}{\R}{\R}, \qquad m_0 = \big(1-\alpha^2{\ts\frac{\md^2}{\md x^2}}\big)u_0.
\end{equation}

The CH equation was introduced by Camassa and Holm~\cite{cammasa1993integrable} as a model for the unidirectional motion of waves at the free surface under the influence of gravity.  The physical relevance of the CH equation is outlined in further detail in References~\cites{johnson2002camassa,dullin2003camassa}.  
Camassa and Holm found travelling-wave solutions of the CH equation of the form 
\begin{equation}
u(x,t)=c\,\me^{-|x-ct|/\alpha},
\label{eq:single}
\end{equation}
where $c$ is a constant wave speed.  Linear combinations of such travelling waves are also solutions of the CH equation, in this case the wave speed of each component wave depends on time~\cite{cammasa1993integrable}.  Such solutions are solitons: the individual components emerge from interactions with their original shapes and speeds intact~\cites{cammasa1993integrable,camassa1994new}.  As the partial derivative of $u$ with respect to $x$ in Equation~\eqref{eq:single} does not exist at $x=0$,  such soliton solutions have to be understood as weak solutions of the CH equation. These are referred to in the literature as \textit{peakons}, or \textit{particle solutions}.    

The weak solutions of the CH equation have been analyzed rigorously in Reference~\cite{cm:00}.  There, Constantin and Molinet introduce a weak formulation of the CH equation and prove the existence of global weak solutions for the case when $u_0\in H^1(\R)$ and $m_0\in \M^+(\R)$ (the space of positive Radon measures).  The present work is also concerned with weak solutions, the aim here is to construct global weak solutions of the CH equation by applying a metric Arzel\`a-Ascoli compactness result to a space of measure-valued functions. 
Before doing this, in the remainder of the introduction we formulate the weak form of the CH equation to be used in this article, as well as describing in detail the properties of the peakons.  We also place our work in the context of the existing literature on the subject, and introduce the necessary notation.

\subsection*{Weak solution of the Camassa--Holm Equation}

We say that $u$ is a weak solution of \ref{eqn_CH}, subject to \eqref{eqn_init_condition}, if $u(0)=u_0$ and for all test functions $\phi \in C^\infty_c(\Omega)$ we have
\begin{align}
 \lint{-\infty}{\infty}{\phi(0,x)}{m_0(x)} &+ \lint{0}{\infty}{\lint{-\infty}{\infty}{(\phi_t - \alpha^2 \phi_{txx})u}{x}}{t} \nonumber\\
 &+ \lint{0}{\infty}{\lint{-\infty}{\infty}{\ts(\frac{3}{2}\phi_x - \frac{1}{2}\alpha^2 \phi_{xxx})u^2}{x}}{t} \nonumber\\
 &+ \lint{0}{\infty}{\lint{-\infty}{\infty}{\ts\frac{1}{2}\alpha^2 \phi_x u^2_x}{x}}{t} = 0. \label{eqn_weak_CH}
\end{align}

\subsection*{Peakons}

Consider an initial condition for the CH equation
\begin{equation}
u_0(x)=\sum_{i=1}^N p_i^0 G(x-x_i^0),\qquad m_0(x)=\sum_{i=1}^N p_i^0 \delta(x-x_i^0),
\end{equation}
where the  positions $x_i^0$ and weights $p_i^0$ and are constants, and where
\begin{equation}
G(x)=\frac{1}{2\alpha}\me^{-|x|/\alpha},\qquad x\in\R
\label{eq:GHelm}
\end{equation}
is the Green's function for the modified Helmholtz problem $(1-\alpha^2\partial_{xx})G(x)=\delta(x)$; this is precisely the modified Helmholtz problem that appears in \ref{eqn_CH}.  Then, a trial solution of the CH equation is:
\begin{equation}
u^{(N)}(t,x)=\sum_{i=1}^N p_i(t) G(x-x_i(t)),\qquad m^{(N)}(t,x)=\sum_{i=1}^N p_i(t) \delta(x-x_i(t)),\qquad t\geqslant 0,
\label{eq:utrial}
\end{equation}
where the positions $\{x_i(t)\}_{i=1}^N$ and  weights $\{p_i(t)\}_{i=1}^N$ are functions of time to be determined.  Evidently, we require:
\[
x_i(0)=x_i^0,\qquad p_i(0)=p_i^0,\qquad i\in\{1,2,\ldots,N\}.
\] 
Furthermore, we substitute 
Equation~\eqref{eq:utrial} into the weak formulation~\eqref{eqn_weak_CH}.  There is then a well-established procedure~\cite{holm2008geometric} that yields a set of ordinary differential equations (ODEs) for the positions and weights, valid for $t>0$:
\begin{equation}\label{eqn_ODEs}
\begin{cases}
\dot{x}_j(t) &= u^{(N)}(t,x_j(t)) = {\ds \frac{1}{2\alpha}\sum_{i=1}^N p_i(t)\me^{-|x_j(t)-x_i(t)|/\alpha}}\\
\dot{p}_j(t) &= {\ds \frac{1}{2\alpha^2}p_j(t)\sum_{i\neq j} p_i(t) \sgn(x_j(t)-x_i(t))\me^{-|x_j(t)-x_i(t)|/\alpha}.}
\end{cases}
\end{equation}
Observe that the expression for $\dot{p}_j(t)$ in \eqref{eqn_ODEs} can be written alternatively as
\begin{equation}\label{eqn_p_j_alternative}
 \dot{p}_j(t) = \ts{\frac{1}{2}p_j(t)\big[u^{(N)}_x(t,x_j(t)_-)+u^{(N)}_x(t,x_j(t)_+)\big].}
\end{equation}
This follows because $G'(0_-)+G'(0_+)=0$ (note that we omit the index $j$ in the sum defining $\dot{p}_j(t)$).

Equation~\eqref{eqn_ODEs} are the evolution equations which govern the positions $\{x_i(t)\}_{i=1}^N$ and the weights $\{p_i(t)\}_{i=1}^N$.  The pairs $(x_i(t),p_i(t))$ are referred to as peakons; also, as Equation~\eqref{eqn_ODEs} is in fact a Hamiltonian system~\cite{cammasa1993integrable}, the peakons can also be identified with notional particles, with $x_i(t)$ playing the role of a particle position and $p_i(t)$ playing the role of a particle momentum.

The properties of Equation~\eqref{eqn_ODEs} have implications for the weak global solutions of the Camassa--Holm equation, and for that reason, these properties are reviewed here.  The equations can be written in Hamiltonian form~\cite{cammasa1993integrable}, with Hamiltonian
\begin{equation}
H=\tfrac{1}{2}\sum_{i=1}^N\sum_{j=1}^N p_ip_j\left(\tfrac{1}{2\alpha}\me^{-|x_i-x_j|/\alpha}\right),
\end{equation}
In this context, we make the identification $\mathrm{sgn}(0)=0$; then, Equation~\eqref{eqn_ODEs} can be re-written as
\begin{equation}\label{eqn_Hamiltonian}
\begin{cases}
 \dot{x}_j(t) &= \frac{\partial H}{\partial p_j}\\
 \dot{p}_j(t) &= -\frac{\partial H}{\partial x_j}.
\end{cases}
\end{equation}
Consequently, the finite-dimensional set of equations~\eqref{eqn_Hamiltonian} is completely integrable with an infinite set of constants of motion $\{H_n\}_{n=1}^\infty$ generated by taking the trace of successive powers of the
matrix $L$,
\[
H_n=\mathrm{Tr}(L^n),\qquad n\geqslant 1,
\]
where the $ij^{\text{th}}$ component of $L$ is given by
\[
L_{ij} = p_j\left(\tfrac{1}{2\alpha}\me^{-|x_i-x_j |/2\alpha}\right).
\]
The conserved quantity $H_1$ is proportional to the total momentum of the system, denoted by $P$:
\begin{equation}\label{eqn_conserved_momentum}
H_1=\frac{1}{2\alpha} \sum_{i=1}^N p_i=\frac{1}{2\alpha}P.
\end{equation}
Similarly $H_2$ is proportional to the Hamiltonian, with $H_2=H/\alpha$.

Camassa and co-authors~\cite{camassa2006integral} have analyzed the regularity properties of the ODE system~\eqref{eqn_Hamiltonian}.  The right-hand side of the system can be identified with a vector field in $\R^{2N}$. The vector field is Lipschitz continuous in the domain 
\[
D := \set{x \in \R^N}{x_1 < \cdots < x_N} \times \R^N,
\]
and by the Picard-Lindel\"of Theorem, the initial value problem with 
\[
\left(x_1^0,\ldots,x_N^0,p_1^0,\ldots,p_N^0\right)\in D
\]
has a unique local solution that is $C^1$-continuous with respect to $t$. The global solution is guaranteed if the initial value is in the domain
\[
\tilde{D} := \set{x \in \R^N}{x_1 < \cdots < x_N} \times \set{p\in\R^N}{p_i > 0},
\]
and the solution of Equation~\eqref{eqn_Hamiltonian} obeys the bounds $x_i^0 \leqslant  x_i(t) \leqslant H_1 t + x_i^0$, and $p_i^0\me^{-H_1t/\alpha} \leqslant p_i(t) \leqslant p_i^0\me^{H_1t/\alpha}$ for all $t\geqslant0$.

Furthermore, the solution is contained in $\tilde{D}$ for all $t \geqslant 0$. This proves that no particle collision can occur in finite time, that is, $x_i(t) \neq  x_j(t)$ for all $i\neq j$ and $t \geqslant 0$.

In this work, we focus on the case where the initial momenta $\{p_i^0\}_{i=1}^N$ are all non-negative, this corresponds to right-travelling particles.  It also corresponds to an initial condition $m_0\in\M^+(\R)$.    This guarantees that the solution to the ordinary differential equations satisfied by the peakons (Equation~\eqref{eqn_ODEs}) is globally defined.  This regularity forms a key part of our results concerning the existence of weak solutions of the CH equation.  We comment briefly also on the case where the initial momenta are not all non-negative.  For the two-particle case, this gives rise to a `head-on' collision, and a finite-time singularity in the corresponding system of ODEs~\cite{cammasa1993integrable}.  A similar conclusion applies to the $N$-particle case: collisions occur pairwise between neighbouring right-travelling and left-travelling particles~\cite{beals2001peakon} -- so-called peakons and anti-peakons.  The solution of the ODE system can be restarted after the collision, using either a dissipation rule or a conservation-of-energy rule~\cite{bressan2007global}.

\subsection*{The work in the context of the existing literature} The authors of \cites{clp:12,clp:12a} introduce a convergence analysis to the particle method, based on functions having bounded variation, to demonstrate the existence of weak solutions of a class of PDEs including \ref{eqn_CH}. The functions arising from the particle method are shown to have uniformly pointwise bounded variation in the spatial direction $x$ and satisfy a Lipschitz property in time $t$. They use a compactness theorem \cite{bressan:00}*{Theorem 2.4}, which applies to functions having these properties, to extract a convergent subsequence of particle solutions of \eqref{eqn_init_condition}, whose limit $u$ satisfies \eqref{eqn_weak_CH} subject to \eqref{eqn_init_condition}. In addition (using a natural alternative way to express $u$ -- see below), they establish a regularity property of the solution by showing that $u$ belongs to $C_b(\R^+,H^1(\R))$, that is, the space of bounded continuous functions on $\R^+$ taking values in the Hilbert-Sobolev space $H^1(\R)$.

The present authors observed that the existence of weak solutions having the same regularity properties can be achieved in a more straightforward manner by applying a metric Arzel\`a-Ascoli compactness result to a space of measure-valued functions equipped with the bounded Lipschitz metric. This metric is of Kantorovich–Rubinstein-Wasserstein type; the use of such metrics in connection with evolutionary PDEs and limits of many-particle systems dates back to e.g. \cites{braunkepp1977,dobrushin1979} and has been developed significantly since then (see e.g. \cite{carrillo2020lipschitz} and references therein). In particular, the Camassa-Holm equation is no stranger to analysis via `Lipschitz metrics'. Other authors have already looked at constructing a metric to measure the distance between different weak solutions, either in characteristic space~\cite{grunert2011lipschitz} or using a transformed version of the Camassa--Holm equation~\cite{carrillo2020lipschitz}. In the authors' opinion, the advantages of the approach presented herein, versus the one given in \cites{clp:12,clp:12a}, are that the computations of the estimates needed to secure the existence and desired regularity of solutions are much shorter, and that it can be generalized with minimal computational effort to much more involved problems that also involve particle solutions, for instance, the fourth-order Geometric Thin-Film Equation~\cites{holm2020gdim,ops2023:gtfe}.

\subsection*{Plan of the Paper and notation}

We spend the rest of this introduction outlining the plan of the paper and the main notation. In \Cref{sect_Green} we introduce our family of Green's functions and their associated integral operators. In \Cref{sect_lsc} we show that the particle solutions behave well with respect to the bounded Lipschitz metric on $B_{\M^+(\R)}$ (the set of positive Radon measures having total variation at most $1$). In \Cref{sect_regular} we build the tools needed to show that our solutions satisfy the desired regularity. Finally, in \Cref{sect_exist}, we put the pieces together and show that our techniques can be applied to solve \ref{eqn_CH}. In terms of results, the general approach is summarised and highlighted in \Cref{prop_T_OK,prop_reg_solutions,prop_convergence_L1loc}, and \Cref{cor_test_convergence}. We conclude the paper by proving \Cref{thm_CH_solution}, which demonstrates that \ref{eqn_CH} can be solved using these techniques.

Given $t \in \R^+$, we shall denote by $\map{u(t)}{\R}{\R}$ the function $u(t,\cdot)$ and by $m(t)$ the element of the space $\M(\R)$ of signed Radon measures on $\R$, defined by
\begin{equation}
 m(t) = \big(1-\alpha^2{\ts\frac{\md^2}{\md x^2}}\big)u(t),
\label{eq:mt}
\end{equation}
(where the derivative is taken in the sense of distributions). In this way, we can identify $u$ and $m$ as functions from $\R^+$ into a suitable function and measure space, respectively. Accordingly, Banach space theory provides the theoretical framework for the present study, and we finish the section by presenting the relevant notation.

Given a Banach space $X$ we denote by $B_X$ its closed unit ball. We denote by $C_0(\R)$ the space of continuous real-valued functions $f$ on $\R$ such that $\lim_{|x|\to\infty}f(x)=0$, equipped with the supremum norm $\pndot{\infty}$. Its dual space identifies isometrically with $\mathcal{M(\R)}$ (equipped with total variation norm $\pndot{1}$), which via the standard duality can be expressed as
\[
\pn{\mu}{1} = \sup\set{\mu(f) := \lint{-\infty}{\infty}{f(x)}{\mu(x)}}{f \in B_{C_0(\R)}}, \qquad \mu \in \mathcal{M(\R)}.
\]
The positive part of the dual unit ball $B_{\M(\R)}$ shall be denoted $B_{\M^+(\R)}$. The Dirac evaluation measure at $x \in \R$ shall be denoted $\delta_x$ (in distributional terms this is $t \mapsto \delta(t-x)$). We denote by $|\mu|$ the total variation measure of $\mu \in \M(\R)$. The weak$^*$-topology on $\M(\R)$, denoted $w^*$, is the weakest topology with respect to which all linear maps of the form $\mu \mapsto \mu(f)$, $f \in C_0(\R)$, are continuous. Since we will be working often with measures other than Lebesgue measure on $\R$, for clarity we emphasise that the terms `almost everywhere' (a.e.) and `almost all' (a.a.) will be taken with respect to Lebesgue measure.

The space of compactly supported functions having derivatives of all orders on a subset $\Omega$ of Euclidean space is denoted $C^\infty_c(\Omega)$. We denote by $\pndot{p}$ the usual norm on $L^p(\R)$, $p=1,2$, and $\ip{\cdot}{\cdot}_2$ shall denote the usual inner product on $L^2(\R)$. The standard norm and inner product on the Hilbert-Sobolev space $H^n(\R)$ will be denoted simply by $\ndot$ and $\ip{\cdot}{\cdot}$, respectively. Operator norms will also be denoted by $\ndot$. Since $H^n(\R)$ is a real Hilbert space we can and will identify it with its dual in the standard way. The Lipschitz constant of a Lipschitz function $g$ will be denoted $\Lip(g)$.

A function $\map{f}{\R}{\R}$ having finite pointwise bounded variation shall be called a BV function, with said variation denoted $\var(f)$. The Banach space of (equivalence classes of) integrable BV functions is denoted $BV(\R)$. This space is equipped with norm $\pn{f}{BV} := \pn{f}{1} + \essvar(f)$, where $\essvar$ denotes essential variation. We shall interpret $f \in BV(\R)$ to mean that $f$ is the (unique) left-continuous representative of its corresponding equivalence class, whose variation matches the latter's essential variation. The theory of BV functions makes permissible this slight abuse of notation. Making this assumption about $f$ means that we can write $f(x)=\mu(-\infty,x)$ for all $x \in \R$, where $\mu \in \M(\R)$ is the distributional derivative of $f$ and satisfies $\pn{\mu}{1}=\var(f)$.

\section{Green's functions and integral operators}\label{sect_Green}

We will adopt the following general scheme for specifying Green's functions of sufficient regularity for our purposes. It is a generalisation of that found on \cite{clp:12a}*{p.~3}. Fix $n \in \N$ and let $\map{G}{\R}{\R}$ be an even function, such that the classical derivative $G^{(k)}$ of order $k$ is defined everywhere, absolutely continuous and integrable on $\R$ for $0 \leqslant k < n$, and there exists $F \in BV(\R)$ that agrees with $G^{(n)}$ a.e.~on $\R$. A straightforward argument demonstrates that $G^{(n)}(x)$ must exist whenever $F$ is continuous, which is at all but at most countably many $x \in \R$. Occasionally we will refer to $\pn{G^{(k)}}{\infty}$; in the case $k=n$ we interpret this as essential supremum.

With a slight abuse of notation we can write $G \in W^{n,1}(\R)$ and $G^{(n)} \in BV(\R)$. It follows that $G^{(k)}$ is essentially bounded for all $k \leqslant n$ and thus $G \in H^n(\R)$ as well.

\begin{example}\label{ex_CH_Green}
In the particular case of \ref{eqn_CH} $G$ is the Green's function for the modified Helmholtz problem $(1-\alpha^2\partial_{xx})G(x)=\delta(x)$, hence, $G$ is given by Equation~\eqref{eq:GHelm}.  In this case, $n=1$, $G$ is absolutely continuous and integrable on $\R$, and there exists an integrable BV function that agrees with $G'$ at all non-zero points.
\end{example}

Given $\mu \in \M(\R)$, straightforward arguments show that for $k < n$, $G^{(k)} * \mu$ is absolutely continuous with derivative $G^{(k+1)} * \mu$, everywhere if $k<n-1$ and almost everywhere if $k=n-1$ (a more delicate analysis can yield a stronger conclusion in the latter case, but it is not required here). Certainly, this implies
\begin{equation}\label{eqn_conv_deriv}
(G * \mu)^{(k)} = G^{(k)} * \mu \qquad\text{a.e.~on $\R$ and all }k \leqslant n. 
\end{equation}

Given a map $\map{f}{\R}{\R}$ and $x \in \R$, define $\map{f_{(x)}}{\R}{\R}$ by $f_{(x)}(t)=f(x-t)$, $t \in \R$. We observe that if $f \in H^n(\R)$ then so is $f_{(x)}$ because
\[
(f_{(x)})^{(k)}(t) = (-1)^kf^{(k)}(x-t) = (-1)^k(f^{(k)})_{(x)}(t), \qquad \text{a.a.~}t \in \R\text{ and all }k \leqslant n, 
\]
and
\[
 \n{f_{(x)}}^2 = \sum_{k=1}^n \lint{-\infty}{\infty}{f^{(k)}(x-t)^2}{t} = \sum_{k=1}^n \lint{-\infty}{\infty}{f^{(k)}(x)^2}{x} = \n{f}^2, \quad\text{$x \in \R$.}
\]

Now define a linear map $\map{T}{H^n(\R)}{C_0(\R)}$ by 
\[
 (Tu)(x) = \sum_{k=0}^n (-1)^k(G^{(k)} * u^{(k)})(x) = \sum_{k=0}^n \lint{-\infty}{\infty}{(G_{(x)})^{(k)}(t)u^{(k)}(t)}{t} = \ip{G_{(x)}}{u}.
\]

Observe that the next result, which will be of most help in \Cref{sect_regular}, only requires the property that $G \in H^n(\R)$.

\begin{proposition}\label{prop_T_OK} We verify the following facts about $T$:
\begin{enumerate}
\item $T$ takes values in $C_0(\R)$, so is well-defined;
\item $T$ is bounded and $\n{T} = \n{G}$;
\item $T^*\mu = G * \mu$ for all $\mu \in \mathcal{M(\R)}$, where $\map{T^*}{\M(\R)}{H^n(\R)^* \equiv H^n(\R)}$ is the dual map.
\end{enumerate}
\end{proposition}

\begin{proof}

We verify that $Tu \in C_0(\R)$. First, given $x \in \R$ we have
\begin{align*}
|(Tu)(x+h)-(Tu)(x)|&= \big|\ip{G_{(x+h)}-G_{(x)}}{u}\big|\\
&\leqslant \n{G_{(x+h)}-G_{(x)}}\n{u} \to 0 \quad\text{as $h \to 0$,}
\end{align*}
by the density of $C^\infty_c(\R)$ in $H^n(\R)$ and the translation invariance of $\ndot$. Hence $Tu$ is continuous.

Second, we show that $(Tu)(x)\to 0$ as $|x|\to\infty$. Given $\ep>0$, fix $M > 0$ large enough such that
\[
 \sum_{k=0}^n \pn{G^{(k)}\cdot\ind{\R\setminus J}}{2}^2,\; \sum_{k=0}^n \pn{u^{(k)}\cdot\ind{\R\setminus J}}{2}^2 \leqslant \ep^2,
\]
where $\ind{E}$ denotes the indicator function of $E \subseteq \R$, $J:=[-M,M]$ and $\cdot$ denotes pointwise product. Given $|x| > 2M$, it follows that
\begin{align*}
 |(Tu)(x)| &\leqslant \bigg|\sum_{k=0}^n \ip{(G_{(x)})^{(k)}\cdot\ind{J}}{u^{(k)}}_2 \bigg| + \bigg|\sum_{k=0}^n \ip{(G_{(x)})^{(k)}\cdot\ind{\R\setminus J}}{u^{(k)}}_2 \bigg|\\
 &= \bigg|\sum_{k=0}^n \ip{(G_{(x)})^{(k)}\cdot\ind{J}}{u^{(k)}}_2 \bigg| + \bigg|\sum_{k=0}^n \ip{(G_{(x)})^{(k)}}{u^{(k)}\cdot\ind{\R\setminus J}}_2 \bigg|\\
 &\leqslant \bigg(\sum_{k=0}^n \pn{(G_{(x)})^{(k)}\cdot\ind{J}}{2}^2 \bigg)^{\frac{1}{2}} \n{u} + \n{G_{(x)}} \bigg(\sum_{k=0}^n \pn{u^{(k)}\cdot\ind{\R\setminus J}}{2}^2 \bigg)^{\frac{1}{2}}\\
 &\leqslant \bigg(\sum_{k=0}^n \pn{G^{(k)}\cdot\ind{\R\setminus J}}{2}^2 \bigg)^{\frac{1}{2}} \n{u} + \n{G} \bigg(\sum_{k=0}^n \pn{u^{(k)}\cdot\ind{\R\setminus J}}{2}^2 \bigg)^{\frac{1}{2}}\\
 &\leqslant (\n{u}+\n{G})\ep.
\end{align*}
The penultimate inequality above follows because $|x|>2M$ implies
\[
 \pn{(G_{(x)})^{(k)}\cdot\ind{J}}{2}^2 = \lint{-M}{M}{G^{(k)}(x-t)^2}{t} \leqslant \lint{|t| > M}{}{G^{(k)}(t)^2}{t} = \pn{G^{(k)}\cdot\ind{\R\setminus J}}{2}^2,
\]
whenever $0 \leqslant k \leqslant n$. Therefore $Tu \in C_0(\R)$, as claimed. Second, since $|(Tu)(x)| \leqslant \n{G_{(x)}}\n{u}=\n{G}\n{u}$ for all $x$, it is clear that $\n{T}\leqslant \n{G}$, and setting $u=G$ yields equality.

Finally, using Fubini's Theorem we observe that
\begin{align*}
(T^*\mu)(u) &= \mu(Tu)\\
&= \lint{-\infty}{\infty}{(Tu)(x)}{\mu(x)} \\
&= \lint{-\infty}{\infty}{\sum_{k=0}^n\lint{-\infty}{\infty}{(-1)^kG^{(k)}(x-t)u^{(k)}(t)}{t}}{\mu(x)} \\
&= \lint{-\infty}{\infty}{\sum_{k=0}^n\bigg(\lint{-\infty}{\infty}{(-1)^kG^{(k)}(x-t)}{\mu(x)}\bigg)u^{(k)}(t)}{t}\\
&= \lint{-\infty}{\infty}{\sum_{k=0}^n\bigg(\lint{-\infty}{\infty}{G^{(k)}(t-x)}{\mu(x)}\bigg)u^{(k)}(t)}{t} \tag*{as $G$ is even}\\
&= \sum_{k=0}^n\lint{-\infty}{\infty}{(G * \mu)^{(k)}(t)u^{(k)}(t)}{t} \tag*{by \eqref{eqn_conv_deriv}}\\
&= \ip{G * \mu}{u}.
\end{align*}
This holds for all such $u$, therefore via the standard duality $T^*\mu = G * \mu$ as claimed.
\end{proof}

\section{The bounded Lipschitz metric and the particle method}\label{sect_lsc}

We will consider the (dual) bounded Lipschitz (or Dudley) norm $\ndot_{BL}$ on $\mathcal{M}(\R)$, defined by
\[
    \n{\mu}_{BL} = \sup\set{\mu(f) := \lint{-\infty}{\infty}{f(x)}{\mu(x)}}{f \in \Xi},
\]
where $\Xi :=\set{f \in C_0(\R)}{\pn{f}{\infty}+\Lip(f) \leq 1}$, together with its associated metric $d$ given by $d(\mu,\nu)=\n{\mu-\nu}_{BL}$, $\mu,\nu \in \mathcal{M}(\R)$. This metric is closely related to the Wasserstein distance $W_1$ (see e.g.~\cite{bogachev2007}*{Section 8.3}), but has the advantage of being defined on all of $\mathcal{M}(\R)$. It is evident that $d$ is a $w^*$-lower semicontinuous metric on $\mathcal{M(\R)}$. Sometimes in the literature $\Xi$ is defined differently, with $\pn{f}{\infty}+\Lip(f)$ replaced by $\max\{\pn{f}{\infty},\Lip(f)\}$, but this change yields a corresponding Lipschitz equivalent norm and metric to which all subsequent arguments apply equally, with only some possible changes to constants required.

For our existence result it will serve us to define the space of functions
\[
 \mathcal{X} = \set{\map{m}{\R^+}{B_{\mathcal{M^+(\R)}}}}{m \text{ is $d$-continuous}}.
\]
An element $m \in \mathcal{X}$ will be called $d$-Lipschitz if there is a constant $L>0$ such that $d(m(s),m(t)) \leqslant L|s-t|$ for all $s,t \in \R^+$.

The rest of this section is devoted to establishing a relationship between this metric and particle solutions of \ref{eqn_CH}. For this we follow \cite{clp:12} in part. We assume that $G$ is as in \Cref{ex_CH_Green}. Following \cite{clp:12}*{(2.1)}, define $\map{m^{(N)}}{\R^+}{\mathcal{M(\R)}}$, $N \in \N$, by
\begin{equation}
m^{(N)}(t) = \sum_{i=1}^N p_i(t) \delta_{x_i(t)},
\label{eq:mNdef}
\end{equation}
where $x_i(t)$ and $p_i(t)$ give the position and mass of the $j$th of $N$ particles, respectively. The $x_i$ and $p_i$ also depend on $N$ but we will suppress this dependency. Equation~\eqref{eq:mNdef} is simply a restatement of the particle solutions in Equation \eqref{eq:utrial}, but using different notation to reflect the fact that the delta functions are elements of a Banach space.

Now define $\map{u^{(N)}}{\R}{\R}$, $t \in \R^+$, by $u^{(N)}(t) = T^*m^{(N)}(t) = G * (m^{(N)}(t))$. Observe that, with the identification $u^{(N)}(t,x)=u^{(N)}(t)(x)$, this definition of $u^{(N)}$ agrees with the one given in \eqref{eq:utrial}. We want the $u^{(N)}$ to satisfy \eqref{eqn_weak_CH} (with initial conditions $u^{(N)}(0)$ and $m^{(N)}(0)$); in order for this to be true, it is necessary and sufficient that, for each $N$, the $x_i$ and $p_i$ obey the system of ODEs in Equation~\eqref{eqn_ODEs}.

The idea is that as $N$ increases the solutions $u^{(N)}$ give a better approximation of a solution of \eqref{eqn_weak_CH}, subject to our initial condition \eqref{eqn_init_condition}. To this end, we arrange $x_i(0)$ and $p_i(0)$ (which again depend also on $N$) in such a way that
\begin{equation}\label{eqn_approx_init}
m^{(N)}(0) \stackrel{w^*}{\to} m_0 \qquad\text{as }N\to\infty. 
\end{equation}
We will assume that $\pn{m_0}{1}=1$ (if not we can rescale). According to \eqref{eqn_conserved_momentum}, the total momentum of the particle system is conserved. Therefore
\begin{equation}\label{eqn_mom}
\sum_{i=1}^N p_i(t) = \pn{m^{(N)}(t)}{1} = 1 \qquad\text{for all }t \in \R^+.
\end{equation}

\begin{proposition}\label{prop_unif_Lipschitz}
We have $m^{(N)} \in \mathcal{X}$ for all $N \in \N$. Moreover, each $m^{(N)}$, $N \in \N$, is $d$-Lipschitz, with $\sup_N \Lip(m^{(N)}) < \infty$.
\end{proposition}

\begin{proof}
As the first statement of the proposition follows immediately from the second, it is sufficient to prove the second statement only.

Given $s,t \in \R^+$, we have
\begin{equation}\label{eqn_unif_Lipschitz_1}
m^{(N)}(s) - m^{(N)}(t) = \sum_{i=1}^N p_i(s)(\delta_{x_i(s)}-\delta_{x_i(t)}) + (p_i(s)-p_i(t))\delta_{x_i(t)}.
\end{equation}

Let $f \in \Xi$. From \eqref{eqn_unif_Lipschitz_1} we have
\begin{align}
& |(m^{(N)}(s) - m^{(N)}(t))(f)| \nonumber\\
=& \bigg| \sum_{i=1}^N p_i(s)(f(x_i(s))-f(x_i(t))) + (p_i(s)-p_i(t))f(x_i(t)) \bigg| \nonumber\\
\leqslant& \sum_{i=1}^N p_i(s)|x_i(s)-x_i(t)| + \sum_{i=1}^N |p_i(s)-p_i(t)|. \label{eqn_unif_Lipschitz_2}
\end{align}

In view of \eqref{eqn_mom} and the definition of $u^{(N)}$, it is clear that $\pn{u^{(N)}}{\infty} \leqslant \pn{G}{\infty}$, and furthermore $|u^{(N)}_x(t,x)| \leqslant \pn{G'}{\infty}$ whenever $u^{(N)}_x(t,x)$ is defined, i.e.~whenever $x$ is distinct from $x_i(t)$, $1 \leqslant i \leqslant n$. With this in mind, together with \eqref{eqn_p_j_alternative}, we arrive at
\begin{equation}\label{eqn_unif_Lipschitz_3}
|\dot{x}_i(t)| \leqslant \pn{G}{\infty} \quad\text{and}\quad |\dot{p}_i(t)| \leqslant \pn{G'}{\infty}|p_i(t)| \qquad\text{for }0 \leqslant i \leqslant n.
\end{equation}

From \eqref{eqn_ODEs}, \eqref{eqn_mom}, \eqref{eqn_unif_Lipschitz_3} and the Mean Value Theorem we have
\begin{equation}\label{eqn_unif_Lipschitz_4}
|x_i(s)-x_i(t)| \leqslant \pn{G}{\infty}|s-t| \qquad\text{for $1 \leqslant i \leqslant n$}
\end{equation}
and
\begin{equation}\label{eqn_unif_Lipschitz_5}
\sum_{i=1}^N|p_i(s)-p_i(t)| \leqslant \pn{G'}{\infty}|s-t|.
\end{equation}

Therefore, from \eqref{eqn_mom}, \eqref{eqn_unif_Lipschitz_2}, \eqref{eqn_unif_Lipschitz_4} and \eqref{eqn_unif_Lipschitz_5} we have
\begin{equation}\label{eqn_unif_Lipschitz_6}
|(m^{(N)}(s) - m^{(N)}(t))(f)| \leqslant (\pn{G}{\infty} + \pn{G'}{\infty})|s-t|.
\end{equation}

Taking the supremum of \eqref{eqn_unif_Lipschitz_6} over $f \in \Xi$ yields
\begin{equation}\label{eqn_unif_Lipschitz_7}
d(m^{(N)}(t),m^{(N)}(s)) = \pn{m^{(N)}(s) - m^{(N)}(t)}{BL} \leqslant (\pn{G}{\infty} + \pn{G'}{\infty})|s-t|.
\end{equation}
This completes the proof.
\end{proof}

\begin{remark}
We remark that as only a few general details are demanded of the ODE system \eqref{eqn_ODEs} for the conclusion of \Cref{prop_unif_Lipschitz} to hold, this proposition can be adapted very easily to suit a range of other systems; see e.g. \cite{ops2023:gtfe}*{Theorem 4.8}, or the proof that the functions $u^N$ are Lipschitz in time in \cite{clp:12a}*{Theorem 2.8}.
\end{remark}

\section{Getting sufficiently regular solutions}\label{sect_regular}

We would like to obtain weak solutions $u$ of \ref{eqn_CH} that lie in $C^{0,\frac{1}{2}}_b(\R^+;H^1(\R))$, which is the space of $\frac{1}{2}$-H\"older continuous functions $\map{u}{\R^+}{H^1(\R)}$ that are also bounded in the sense that $\sup_{t \in \R^+}\n{u(t)} < \infty$. In this section we develop a general tool that will guarantee this regularity in the case of \ref{eqn_CH} and other situations. For this we require a series of lemmas. The first lemma is elementary; we couldn't find a proof of it in the literature so one is provided for completeness.

\begin{lemma}\label{lem_reg_kernel}
Let $F \in BV(\R)$. Then, given $t \in \R$, we have
\[
\lint{-\infty}{\infty}{|F(x+t)-F(x)|}{x} \leqslant \var(F)|t|.
\]
\end{lemma}

\begin{proof}
Let $\mu \in \mathcal{M}(\R)$ such that $\pn{\mu}{1}=\var(F)$ and $F(x)=\mu(-\infty,x)$ for all $x \in \R$. Given $t > 0$, we have $F(x+t)-F(x)=\mu[x,x+t)$, $x \in \R$. Define $\map{w}{\R^2}{\R}$ by $w(x,y)=1$ if $x \leqslant y < x+t$ and $w(x,y)=0$ otherwise. By Fubini's Theorem we have
\begin{align*}
\lint{-\infty}{\infty}{|F(x+t)-F(x)|}{x} \leqslant& \lint{-\infty}{\infty}{\lint{-\infty}{\infty}{w(x,y)}{|\mu|(y)}}{x}\\
=& \lint{-\infty}{\infty}{\lint{-\infty}{\infty}{w(x,y)}{x}}{|\mu|(y)}\\
=& \lint{-\infty}{\infty}{|t|}{|\mu|(y)} = \pn{\mu}{1}|t| = \var(F)|t|.
\end{align*}
A similar argument applies when $t<0$.
\end{proof}

The next lemma yields a Young-type inequality in terms of the norm $\pndot{BL}$.

\begin{lemma}\label{lem_Lip_1}
Let $F \in BV(\R)$. Then
\[
\pn{F * \mu}{1} \leqslant \pn{F}{BV}\pn{\mu}{BL} \qquad\text{for all }\mu \in \mathcal{M(\R)}.
\]
\end{lemma}

\begin{proof}
Let $\mu \in \mathcal{M(\R)}$. First observe that
\begin{equation}\label{eqn_Tonelli_1}
\lint{-\infty}{\infty}{\lint{-\infty}{\infty}{|F(x-y)|}{x}}{|\mu|(y)} = \pn{F}{1}\pn{\mu}{1} < \infty.
\end{equation}

Set $s(x)=\sgn(F * \mu)(x)$, $x \in \R$. Given \eqref{eqn_Tonelli_1}, we can apply Fubini's Theorem to obtain 
\begin{align}
\pn{F * \mu}{1} =& \lint{-\infty}{\infty}{|(F * \mu)(x)|}{x} \nonumber\\ 
=& \lint{-\infty}{\infty}{\lint{-\infty}{\infty}{s(x)F(x-y)}{\mu(y)}}{x} = \lint{-\infty}{\infty}{\tilde{F}(y)}{\mu(y)}, \label{eqn_Lip_1}
\end{align}
where $\map{\tilde{F}}{\R}{\R}$ is given by
\[
\tilde{F}(y) = \lint{-\infty}{\infty}{s(x)F(x-y)}{x}.
\]
Now $\|\tilde{F}\|_\infty \leqslant \pn{F}{1}$ and by \Cref{lem_reg_kernel} we have
\begin{equation}\label{eqn_Lip_1a}
|\tilde{F}(y)-\tilde{F}(z)| \leqslant \lint{-\infty}{\infty}{|F(x-y)-F(x-z)|}{x} \leqslant \var(F)|y-z|.
\end{equation}
Hence $\tilde{F}$ is a bounded Lipschitz function. However, to finish the proof we need to consider functions in $C_0(\R)$, and $\tilde{F}$ may not be an element of this space.

To this end, let $\midd\{a, b, c\}$ denote the number $a$, $b$ or $c$ that lies (not necessarily strictly) between the other two, and extend this notation pointwise to functions. Given $k \in \N$, define the Lipschitz maps $\map{H_k,\tilde{F}_k}{\R}{\R}$ by $H_k(y)=\max\{\var(F)(k-|y|),0\}$ and $\tilde{F}_k=\midd\{-H_k,\tilde{F},H_k\}$. Then $\tilde{F}_k \in C_0(\R)$, $\|\tilde{F}_k\|_\infty \leqslant \|\tilde{F}\|_\infty \leqslant \pn{F}{1}$ and $\Lip(\tilde{F}_k) \leqslant \var(F)$ for all $k \in \N$. Hence by this and \eqref{eqn_Lip_1a} we have
\[
\|\tilde{F}_k\|_\infty + \Lip(\tilde{F}_k) \leqslant \pn{F}{BV} \qquad\text{for all }k \in \N.
\]
Consequently
\[
\lint{-\infty}{\infty}{\tilde{F}_k(y)}{\mu(y)} \leqslant \pn{F}{BV}\pn{\mu}{BL} \qquad\text{for all }k \in \N.
\]
Hence by this, \eqref{eqn_Lip_1} and the Dominated Convergence Theorem we obtain
\[
\pn{F * \mu}{1} = \lint{-\infty}{\infty}{\tilde{F}(y)}{\mu(y)} = \lim_{k\to\infty}\lint{-\infty}{\infty}{\tilde{F}_k(y)}{\mu(y)} \leqslant \pn{F}{BV}\pn{\mu}{BL}. \qedhere
\]
\end{proof}

\begin{lemma}\label{lem_Lip_2}
There exists $L>0$ such that
\begin{equation}\label{eqn_Lip_2}
\n{T^*\mu} \leqslant L^{\frac{1}{2}}\pn{\mu}{BL}^{\frac{1}{2}} \qquad\text{for all }\mu \in B_{\mathcal{M(\R)}}.
\end{equation}
\end{lemma}

\begin{proof} Define
\[
 L = \sum_{k=0}^n \var(G^{(k)})\pn{G^{(k)}}{BV}.
\]
By \Cref{lem_Lip_1} we have
\begin{equation}
\|G^{(k)} * \mu\|_1 \leqslant \pn{G^{(k)}}{BV}\pn{\mu}{BL} \qquad\text{for }\mu \in \mathcal{M(\R)}\text{ and }k \leqslant n. \label{eqn_Lip_2_a}
\end{equation}
Given $\mu \in B_{\mathcal{M(\R)}}$ and $k \leqslant n$, we have
\[
\|G^{(k)} * \mu\|_\infty \leqslant \|G^{(k)}\|_\infty \leqslant \var(G^{(k)}),
\]
(where we use essential supremum and variation in the case $k=n$). Hence by this, \eqref{eqn_conv_deriv} and \eqref{eqn_Lip_2_a} we can estimate
\[
\n{T^*\mu}^2 = \sum_{k=0}^n \lint{-\infty}{\infty}{(G^{(k)} * \mu)(x)^2}{x} \leqslant \sum_{k=0}^n \var(G^{(k)})\|G^{(k)} * \mu\|_1 \leqslant L\pn{\mu}{BL}. \qedhere
\]
\end{proof}

We conclude this section with a proposition that will provide us with the desired regularity of solutions. This proposition is also used to show, with minimal effort, that solutions of the Geometric Thin-Film Equation \cite{ops2023:gtfe}*{Theorem 4.8} also have the regularity that we would wish for.

\begin{proposition}\label{prop_reg_solutions}
Let $m \in \mathcal{X}$ and define $\map{u}{\R^+}{H^n(\R)}$ by $u(t)=T^*m(t)$. Then $u \in C_b(\R^+;H^n(\R))$. Moreover, if $m$ is $d$-Lipschitz then $u \in C^{0,\frac{1}{2}}_b(\R^+;H^n(\R))$.
\end{proposition}

\begin{proof}
We observe first that $u$ is well-defined by \Cref{prop_T_OK}, and that
\[
\n{u(t)} \leqslant \n{T^*}\pn{m(t)}{1} \leqslant \n{T^*} = \n{T}\qquad\text{ for all }t \in \R^+,  
\]
so $u$ is bounded with respect to $\ndot$. Now take $L>0$ from \Cref{lem_Lip_2}. Let $s,t \in \R^+$. As $m(s),m(t) \in B_{\mathcal{M(\R)}}$ we have $\frac{1}{2}(m(s)-m(t)) \in B_{\mathcal{M(\R)}}$. By \Cref{lem_Lip_2} we see that
\begin{align*}
\n{u(s) - u(t)} =& \n{T^*(m(s)-m(t))}\\
=& 2\n{T^*{\ts \frac{1}{2}}(m(s)-m(t))}\\
\leqslant& (2L)^{\frac{1}{2}}\pn{m(s)-m(t)}{BL}^{\frac{1}{2}} = (2L)^{\frac{1}{2}}d(m(s),m(t))^{\frac{1}{2}}.
\end{align*}
As $m$ is $d$-continuous, we have $u \in C_b(\R^+;H^n(\R))$. If moreover $m$ is $d$-Lipschitz, then $u \in C^{0,\frac{1}{2}}_b(\R^+;H^n(\R))$.
\end{proof}

\section{Existence of weak solutions}\label{sect_exist}

In this section we draw together the threads of the previous sections to obtain weak solutions of \ref{eqn_CH} having the desired regularity.

We begin by making some assertions that can apply in a more general context before concentrating on the CH case at the end. First, we make use of a metric Arzel\`a-Ascoli existence result. Such results have also been used in \cite{dl:15}*{Section 4} for example.

Recall that $B_{\mathcal{M^+(\R)}}$ is compact in the $w^*$-topology. Given that $d$ is a $w^*$-lower semicontinuous metric on this set, we can readily deduce the following result from \cite{ags:05}*{Proposition 3.3.1}.

\begin{proposition}\label{prop_exist}
Let $(m^{(N)}) \subseteq \mathcal{X}$ be an equicontinuous sequence with respect to $d$. Then there exists a subsequence $(m^{(N_k)})$ and $m \in \mathcal{X}$ such that
\[
m^{(N_k)}(t) \stackrel{w^*}{\to} m(t) \qquad\text{for all }t \in \R^+.
\]
Moreover, if there exists $L>0$ such that $\sup_N\Lip(m^{(N)}) \leqslant L$ with respect to $d$, then $\Lip(m) \leqslant L$. 
\end{proposition}

\begin{proof}
The existence of $(m^{(N_k)}(t)) \subseteq \mathcal{X}$ and $m \in \mathcal{X}$ follows straightaway from \cite{ags:05}*{Proposition 3.3.1}. To see the last part, assume $\sup_N\Lip(m^{(N)}) \leqslant L$ with respect to $d$. Then given $s,t \in \R^+$, by the $w^*$-lower semicontinuity of $d$ we have
\[
d(m(s),m(t)) \leqslant \liminf_{k\to\infty} d(m^{(N_k)}(s),m^{(N_k)}(t)) \leqslant L|s-t|. \qedhere
\]
\end{proof}

In the next series of results, we assume that $(m^{(N)}) \subseteq \mathcal{X}$ and $m \in \mathcal{X}$ satisfy
\[
m^{(N)}(t) \stackrel{w^*}{\to} m(t)\qquad\text{ for all }t \in \R^+. 
\]
In accordance with \Cref{prop_reg_solutions}, define $u,u^{(N)} \in C_b(\R^+;H^n(\R))$ by
\[
 u(t) = T^*m(t) = G * (m(t)) \quad\text{and}\quad u^{(N)}(t) = T^*m^{(N)}(t) = G * (m^{(N)}(t)),
\]
for $t \in \R^+$ and $N \in \N$. As per the Introduction, we identify $u$ and the $u^{(N)}$ as functions on $\Omega:=\R^+ \times \R$ in the natural way by writing $u(t,x)=u(t)(x)$ and $u^{(N)}(t,x)=u^{(N)}(t)(x)$, $(t,x) \in \Omega$.

The next step is to show that $\partial^k_x u^{(N)} \to \partial^k_x u$ in the topology of $L^1_{\mathrm{loc}}(\Omega)$ for $0\leqslant k \leqslant n$. To this end, define $\Omega_R = [0,R] \times [-R,R]$ for all $R>0$. Given a function $\map{f}{\Omega}{\R}$, define $\map{f\restrict{R}}{\Omega}{\R}$ by $f\restrict{R} \;= f\cdot\ind{\Omega_R}$, where $\ind{\Omega_R}$ denotes the indicator function of $\Omega_R$ and $\cdot$ denotes pointwise product. Convergence in $L^1_{\mathrm{loc}}(\Omega)$ follows from the next proposition.

\begin{proposition}\label{prop_convergence_L1loc}
Given $R > 0$ and $k \leqslant n$, we have
\begin{equation}\label{eqn_L1loc_1}
\lim_{N\to\infty}\pn{\partial^k_x u^{(N)}\restrict{R} - \,\partial^k_x u\restrict{R}}{1} = 0. 
\end{equation}
\end{proposition}

We require a lemma in order to prove this result.

\begin{lemma}\label{lem_L1loc_1}
Let $F \in BV(\R)$. Suppose that $(\mu^{(N)}) \subseteq \M(\R)$ satisfies $\mu^{(N)} \stackrel{w^*}{\to} 0$. Then given $R>0$ we have
\[
\lim_{N\to\infty}\lint{-R}{R}{|(F * \mu^{(N)})(x)|}{x} = 0. 
\]
\end{lemma}

\begin{proof}
We assume otherwise and derive a contradiction. We follow the proof of \Cref{lem_Lip_1} to some extent. By the Uniform Boundedness Theorem we can assume without loss of generality that $\pn{\mu^{(N)}}{1} \leqslant 1$ for all $N$. Fix $R>0$. Given $N \in \N$, define $s_N(x)=\sgn(F * \mu^{(N)})(x)$, $x \in \R$. By Fubini's Theorem we have 
\begin{align}
\lint{-R}{R}{|(F * \mu^{(N)})(x)|}{x} =& \lint{-R}{R}{s_N(x)(F * \mu^{(N)})(x)}{x} \nonumber\\ 
=& \lint{-R}{R}{\lint{-\infty}{\infty}{s_N(x)F(x-y)}{\mu^{(N)}(y)}}{x} \nonumber\\
=& \lint{-\infty}{\infty}{\tilde{F}_N(y)}{\mu^{(N)}(y)}, \label{eqn_L1loc_3}
\end{align}
where $\map{\tilde{F}_N}{\R}{\R}$ is given by
\[
\tilde{F}_N(y) = \lint{-R}{R}{s_N(x)F(x-y)}{x}.
\]
We shall also define $\map{\tilde{F}}{\R}{\R}$ by
\[
\tilde{F}(y) = \lint{-R}{R}{|F(x-y)|}{x}.
\]
It is evident that $|\tilde{F}_N(y)| \leqslant \tilde{F}(y)$ and $\|\tilde{F}_N\|_\infty \leqslant \|\tilde{F}\|_\infty \leqslant \pn{F}{1}$ for all $y \in \R$ and $N \in \N$, and similarly to \eqref{eqn_Lip_1a} we see that $\tilde{F}$ and the $\tilde{F}_N$ are Lipschitz functions with $\Lip(\tilde{F}), \Lip(\tilde{F}_N) \leqslant \var(F)$ for all $N \in \N$.

We claim further that $\tilde{F}$, and thus each $\tilde{F}_N$, is an element of $C_0(\R)$. Indeed, let $\ep>0$. As $F \in L^1(\R)$, there exists $M>0$ such that 
\[
 \lint{|x| > M}{}{|F(x)|}{x} \leqslant \ep.
\]
Hence, if $|y| > R+M$ then $|x-y| > M$ whenever $|x| < R$, thus $\tilde{F}(y) \leqslant \ep$ whenever $|y|>R+M$. This proves the claim.

By our assumption and \eqref{eqn_L1loc_3}, there exists $\eta > 0$ and a subsequence of the $\tilde{F}_N$, labelled in the same way, such that
\begin{equation}
 \lint{-\infty}{\infty}{\tilde{F}_N(y)}{\mu^{(N)}(y)} \geqslant \eta \qquad\text{for all }N \in \N. \label{eqn_L1loc_4}
\end{equation}
Since $\sup_N \Lip(\tilde{F}_N) \leqslant \var(F)$, the $\tilde{F}_N$ form an equi-Lipschitz family. Hence by an Arzel\`a-Ascoli argument, there exists a subsequence of the $\tilde{F}_N$, again labelled in the same way, and a Lipschitz function $\map{H}{\R}{\R}$ such that, for all $N \in \N$, the restrictions of the $\tilde{F}_N$ to $[-q,q]$ converge uniformly to the restriction of $H$ to $[-q,q]$.

We claim we have $H \in C_0(\R)$ and uniform convergence of the $\tilde{F}_N$ to $H$ on $\R$, i.e. ${\|\tilde{F}_N-H\|_\infty} \to 0$. Indeed, given $\ep>0$, let $q \in \N$ such that $\tilde{F}(y) \leqslant \frac{1}{2}\ep$ whenever $y > q$. Then $|\tilde{F}_N(y)| \leqslant \frac{1}{2}\ep$ for all $N \in \N$ and $y > q$, which means that $|H(y)| \leqslant \frac{1}{2}\ep$ for all such $y$. Consequently, $|\tilde{F}_N(y)-H(y)| \leqslant \ep$ for these $y$. Now we appeal to the uniform convergence of the restrictions on $[-q,q]$ to conclude that there exists $N_0 \in \N$ such that $|\tilde{F}_N(y)-H(y)| \leqslant \ep$ for all $N \geqslant N_0$ and $y \in \R$.

In particular, there exists $N_0 \in \N$ such that ${\|\tilde{F}_N-H\|_\infty} \leqslant \frac{\eta}{2}$ whenever $N \geqslant N_0$. Since $\pn{\mu^{(N)}}{1} \leqslant 1$ for all $N \in \N$, we deduce from \eqref{eqn_L1loc_4} that 
\[
 \lint{-\infty}{\infty}{H(y)}{\mu^{(N)}(y)} \geqslant {\ts\frac{1}{2}\eta} \qquad\text{whenever }N \geqslant N_0.
\]
However, this contradicts the fact that $\mu^{(N)} \stackrel{w^*}{\to} 0$. This completes the proof.
\end{proof}

\begin{proof}[Proof of \Cref{prop_convergence_L1loc}]
The first step is to establish
\begin{equation}
 \lim_{N\to\infty}\lint{-R}{R}{|\partial^k_x u^{(N)}(t,x) - \partial^k_x u(t,x)|}{x} = 0 \qquad\text{$k \leqslant n$, $t \in \R^+$ and $R>0$.} \label{eqn_L1loc_5}
\end{equation}
Let $t \in \R^+$. By assumption $m^{(N)}(t)-m(t) \stackrel{w^*}{\to} 0$. According to \eqref{eqn_conv_deriv},
\begin{equation}\label{eqn_L1loc_6}
\partial^k_x u(t) = G^{(k)} * (m(t)) \quad\text{and}\quad \partial^k_x u^{(N)}(t) = G^{(k)} * (m^{(N)}(t))  \qquad\text{a.e.~on $\R$,}
\end{equation}
thus \eqref{eqn_L1loc_5} follows by \Cref{lem_L1loc_1}, given that $G^{(k)} \in BV(\R)$ for $0 \leqslant k \leqslant n$.

By \eqref{eqn_L1loc_6} and the fact that the $m(t)$ and $m^{(N)}(t)$ belong to $B_{\M(\R)}$, for $k \leqslant n$ and $N \in \N$ we have
\begin{equation}\label{eqn_L1loc_7}
\pn{\partial^k_x u}{\infty},\,\pn{\partial^k_x u^{(N)}}{\infty} \leqslant \max_{0\leqslant \ell \leqslant n}\pn{G^{(\ell)}}{\infty},
\end{equation}
where $\pndot{\infty}$ stands for essential supremum wherever necessary. Hence we can finish the proof by applying the Fubini and Dominated Convergence Theorems:
\begin{align*}
 \lim_{N\to\infty}\pn{\partial^k_x u^{(N)}\restrict{R} - \,\partial^k_x u\restrict{R}}{1} =& \lim_{N\to\infty}\lint{0}{R}{\lint{-R}{R}{|\partial^k_x u^{(N)}(t,x) - \partial^k_x u(t,x)|}{x}}{t}\\
 =& 0. \qedhere
\end{align*}
\end{proof}

The next result is what we will use to establish the existence of weak solutions.

\begin{corollary}\label{cor_test_convergence}
Let $\psi \in C^\infty_c(\Omega)$, $0 \leqslant k \leqslant n$ and $\ell \in \N$. Then 
\begin{align}
\lim_{N\to\infty}\lint{-\infty}{\infty}{\psi(0,x)}{(m^{(N)}(0))(x)} &= \lint{-\infty}{\infty}{\psi(0,x)}{(m(0))(x)} \label{eqn_w_convergence_1}
\end{align}
and
\begin{align}
& \lim_{N\to\infty}\lint{0}{\infty}{\lint{-\infty}{\infty}{((\partial^k_x u^{(N)}(t,x))^\ell - (\partial^k_x u(t,x))^\ell)\psi(t,x)}{x}}{t} = 0. \label{eqn_w_convergence_2}
\end{align}
\end{corollary}

\begin{proof}
The first limit \eqref{eqn_w_convergence_1} follows directly from the fact that $m^{(N)}(0) \stackrel{w^*}{\to} m(0)$ and the function $x \mapsto \psi(0,x)$ belongs to $C_0(\R)$. To see that \eqref{eqn_w_convergence_2} holds, we appeal again to \eqref{eqn_L1loc_7}, thus giving
\[
\lim_{N\to\infty}\pn{(\partial^k_x u^{(N)})^\ell\restrict{R} - \,(\partial^k_x u)^\ell\restrict{R}}{1} = 0.
\]
for $R>0$, $k \leqslant n$ and $\ell \in \N$, by \Cref{prop_convergence_L1loc}. Since $\psi$ is bounded and has bounded support in $\Omega$, \eqref{eqn_w_convergence_2} follows by picking $R>0$ large enough.
\end{proof}

We are finally in a position to prove \Cref{thm_CH_solution}.

\begin{theorem}\label{thm_CH_solution}
Let $G$ be as in \Cref{ex_CH_Green} and let $(m^{(N)})\subseteq \mathcal{X}$ be the sequence from \Cref{prop_unif_Lipschitz}. Then there exists a subsequence of $(m^{(N)})$, labelled in the same way, and $m \in \mathcal{X}$, such that
\[
m^{(N)}(t) \stackrel{w^*}{\to} m(t) \qquad\text{for all }t \in \R^+.
\]
Moreover, if we define $u$ on $\R^+$ by $u(t)=T^*m(t) = G * (m(t))$, then $u$ is a solution of \eqref{eqn_weak_CH} with initial conditions \eqref{eqn_init_condition}, and $u \in C^{0,\frac{1}{2}}_b(\R^+;H^1(\R))$.
\end{theorem}

\begin{proof}[Proof of \Cref{thm_CH_solution}]
Again, we use the identification $u^{(N)}(t,x)=u^{(N)}(t)(x)$. It is shown in \cite{clp:12}*{Proposition 3.4} that the functions $u^{(N)}$ defined by $u^{(N)}(t)=T^*m^{(N)}(t)$, $t \in \R^+$, $N \in \N$, satisfy \eqref{eqn_weak_CH} with initial data $m^{(N)}(0)$. 

According to \Cref{prop_exist}, there exists a subsequence of $(m^{(N)})$, which we will label in the same way, and $m \in \mathcal{X}$, such that
\[
m^{(N)}(t) \stackrel{w^*}{\to} m(t) \qquad\text{for all }t \in \R^+.
\]
Moreover, by \Cref{prop_unif_Lipschitz} the $m^{(N)}$, $N \in \N$, are $d$-Lipschitz with $\sup_N \Lip(m^{(N)})$ finite. Thus $m$ is also $d$-Lipschitz.

Recalling \eqref{eqn_approx_init}, we have $m(0)=m_0$, because $(m^{(N)}(0))$ converges to both $m_0$ and $m(0)$ in the $w^*$-topology. Define $\map{u}{\R^+}{H^1(\R)}$ by $u(t)=T^*m(t)$. We have $u(0)=T^*m(0)=T^*m_0 = u_0$. Hence $u$ satisfies \eqref{eqn_init_condition}. That $u$ also satisfies \eqref{eqn_weak_CH} follows by \Cref{cor_test_convergence} and the fact that the $u^{(N)}$ also satisfy \eqref{eqn_weak_CH}.

Finally, as $m$ is $d$-Lipschitz, \Cref{prop_reg_solutions} tells us that $u \in C^{0,\frac{1}{2}}_b(\R^+;H^1(\R))$.
\end{proof}

We finish the article with a remark on the convergence of the sequence $(m^{(N)})$. We have demonstrated the existence of solutions of \ref{eqn_CH} by taking a subsequence of a sequence $(m^{(N)}) \subseteq \mathcal{X}$ (that correspond to particle solutions), that converges to a limit $m \in \mathcal{X}$. We know that, given initial Radon data $m_0$, the solution of \ref{eqn_CH} is unique \cite{cm:00}. We can use this fact to show that given the initial data $m_0$, the entire sequence $(m^{(N)})$ converges to $m$ in the following sense.

\begin{proposition}
Let $(m^{(N)}) \subseteq \mathcal{X}$ denote the entire sequence of particle solutions as described above, with $G$ as in \Cref{ex_CH_Green}. Then $m^{(N)}(t) \stackrel{w^*}{\to} m(t)$ for all $t \in \R^+$.
\end{proposition}

\begin{proof}
We assume otherwise and obtain a contradiction. Suppose that there exists $t_0 \in \R^+$ such that 
\[
 m^{(N)}(t_0) \not\stackrel{w^*}{\to} m(t_0).
\]
Then there exists a $w^*$-open subset $U \ni m(t_0)$ of $B_{\M^+(\R)}$ and a subsequence $(m^{(N_k)})$ of $(m^{(N)})$, such that $m^{(N_k)}(t_0) \not\in U$ for all $k \in \N$. We can apply \Cref{prop_exist} to obtain a further subsequence $(m^{(N_{k_i})})$ and $\tilde{m} \in \mathcal{X}$ such that
\[
 m^{N_{k_i}}(t) \stackrel{w^*}{\to} \tilde{m}(t) \qquad\text{for all }t \in \R^+.
\]
Just as above, $\tilde{m}$ will also yield a weak solution of \ref{eqn_CH} having the same initial conditions. However, we have $\tilde{m}(t_0) \notin U$ and in particular $\tilde{m}(t_0) \neq m(t_0)$. (Indeed, as $d$ is $w^*$-lower semicontinuous and thus finer than the $w^*$-topology, and $m$ and $\tilde{m}$ are both $d$-continuous, it follows that $\tilde{m}(t) \neq m(t)$ for all $t$ in an open interval containing $t_0$.) 

Now observe that if $\nu \in \mathcal{M(\R)}$ satisfies $G * \nu=0$, then $\nu=0$. Indeed, taking Fourier transforms yields $0=\hat{G}(s)\hat{\nu}(s)$ for all $s \in \R$, and we can calculate that $\hat{G}(s) = (1+\alpha^2s^2)^{-1} \neq 0$ for all $s \in \R$, whence $\hat{\nu}$ is identically zero, giving $\nu=0$. Therefore, as $m(t_0) \neq \tilde{m}(t_0)$, so $G * (m(t_0)) \neq G * (\tilde{m}(t_0))$. This is a contradiction because it implies two distinct solutions of \ref{eqn_CH}. 
\end{proof}

\subsection*{Acknowledgements}

This publication has emanated from research conducted with the financial support of Science Foundation Ireland under Grant number 18/CRT/6049.
 
\bibliography{ch_bibliography}

\end{document}